\documentclass[12pt]{article}
\usepackage{amsmath,amsthm,amssymb,color,verbatim,graphicx,fullpage,url,cite}
\newcommand{\remove}[1]{}
\newtheorem{theorem}{Theorem}[section]
\newtheorem{claim}[theorem]{Claim}
\newtheorem{lemma}[theorem]{Lemma}
\newtheorem{proposition}[theorem]{Proposition}
\newtheorem{definition}[theorem]{Definition}
\newtheorem{corollary}[theorem]{Corollary}

\newcommand{\R}{\mathbb{R}}
\newcommand{\F}{\mathbb{F}}
\newcommand{\Z}{\mathbb{Z}}
\newcommand{\B}{\mathcal{B}}
\newcommand{\CC}{\mathcal{C}}
\newcommand{\ip}[1]{\langle{#1}\rangle}
\newcommand{\etal}{{et al.}}

\newcommand{\ind}{\textrm{ind}}
\newcommand{\cay}{\text{Cay}}
\newcommand{\Tr}{\textrm{Tr}}
\newcommand{\sign}{\textrm{sign}}

\title{The independence number of the Birkhoff polytope graph, and applications to maximally recoverable codes}
\author{Daniel Kane\thanks{email: dakane@ucsd.edu.}
 \qquad
Shachar Lovett\thanks{email: slovett@ucsd.edu. Research supported by NSF CCF award 1614023 and a Sloan fellowship.}
 \qquad
 Sankeerth Rao\thanks{email: skaringu@ucsd.edu. Research supported by NSF CCF award 1614023.}\\
 Department of Computer Science\\
 University of California, San Diego}

\begin{document}
\maketitle

\begin{abstract}
Maximally recoverable codes are codes designed for distributed storage which combine
quick recovery from single node failure and optimal recovery from catastrophic failure. Gopalan et al [SODA 2017]
studied the alphabet size needed for such codes in grid topologies and gave a combinatorial characterization for it.

Consider a labeling of the edges of the complete bipartite graph $K_{n,n}$ with labels coming from $\F_2^d$,
that satisfies the following condition: for any simple cycle, the sum of the labels over its edges is nonzero.
The minimal $d$ where this is possible controls the alphabet size needed for maximally recoverable codes in $n \times n$ grid topologies.

Prior to the current work, it was known that $d$ is between $(\log n)^2$ and $n \log n$. We improve both bounds and show that
$d$ is linear in $n$. The upper bound is a recursive construction which beats the random construction.
The lower bound follows by first relating the problem to the independence number of the Birkhoff polytope graph,
and then providing tight bounds for it using the representation theory of the symmetric group.
\end{abstract}

\section{Introduction}
The Birkhoff polytope is the convex hull of $n \times n$ doubly stochastic matrices. The Birkhoff polytope graph
is the graph associated with its $1$-skeleton. This graph is well studied as it plays an important role in combinatorics 
and optimization, see for example the book of Barvinok~\cite{barvinok2002course}. For us, this graph arose naturally in the 
study of certain maximally recoverable codes. Our main technical results are tight bounds on the independence number of 
the Birkhoff polytope graph, which translate to tight bounds on the alphabet size needed for maximally recoverable codes in grid topologies.

We start by describing the coding theory question that motivated the current work.

\subsection{Maximally recoverable codes}
Maximally recoverable codes, first introduced by Gopalan, Huang, Jenkins and Yekhanin~\cite{gopalan2014explicit}, are codes designed for distributed storage which combine
quick recovery from single node failure and optimal recovery from catastrophic failure. More precisely, they
are systematic linear codes which combine two types of redundancy symbols:
local parity symbols, which allow for fast recovery from single symbol erasure; and global parity symbols, which allow for recovery
from the maximal information theoretic number of erasures. This was further studied in~\cite{tamo2014family,balaji2015partial,tamo2014bounds,lalitha2015weight,plank2013sd}.

The present paper is motivated by a recent work of
Gopalan, Hu, Kopparty, Saraf, Wang and Yekhanin~\cite{gopalan2017maximally}, which studied the effect of the topology of the network
on the code design. Concretely, they studied grid like topologies. In the simplest setting, a codeword is viewed as an $n \times n$ array,
with entries in a finite field $\F_{2^d}$, where there is a single parity constraint for each row and each column, and an additional global parity constraint. More generally, a
$T_{n \times m}(a,b,h)$ maximally recoverable code has codewords viewed as an $n \times m$ matrix over $\F_2^d$, with $a$ parity constraints per
row, $b$ parity constraints per column, and $h$ additional global parity constraints. An important problem in this context is, how small can we choose the alphabet size $2^d$
and still achieve information theoretical optimal resiliency against erausers.

Gopalan \etal~\cite{gopalan2017maximally} gave a combinatorial characterization for this problem, in the simplest setting of $m=n$ and $a=b=h=1$.
Their characterization is in terms of labeling the edges of the complete bipartite graph $K_{n,n}$ by elements of $\F_2^d$, which satisfy the property
that in every simple cycle, the sum is nonzero.

Let $[n]=\{1,\ldots,n\}$. Let $\gamma:[n] \times [n] \to \F_2^d$ be a labeling of the edges of the complete bipartite graph $K_{n,n}$ by bit vectors of length $d$.

\begin{definition} A labeling $\gamma:[n] \times [n] \to \F_2^d$ is \emph{simple cycle free} if for any simple cycle $C$ in $K_{n,n}$ it holds that
$$
\sum_{e \in C} \gamma(e) \ne 0.
$$
\end{definition}
Gopalan \etal~\cite{gopalan2017maximally} showed that the question on the minimal alphabet size needed for maximally recoverable codes, reduces to the question of how small can we take $d=d(n)$
so that a simple cycle free labeling exists. Concretely:
\begin{itemize}
\item The alphabet size needed for $T_{n \times n}(1,1,1)$ codes is $2^{d(n)}$.
\item The alphabet size needed for $T_{n \times m}(a,b,h)$ codes is at least $2^{\min(d(n-a+1),d(m-b+1)) / h}$.
\end{itemize}

Before the current work, there were large gaps between upper and lower bounds on $d(n)$.
For upper bounds, as the number of simple cycles in $K_{n,n}$ is $n^{O(n)}$, a random construction with $d=O(n \log n)$ succeeds
with high probability. There are also simple explicit constructions matching the same bounds, see e.g.~\cite{gopalan2014explicit}. In terms of lower bounds,
it is simple to see that $d \ge \log n$ is necessary. The main technical lemma of Gopalan \etal~\cite{gopalan2017maximally} in this context is that in fact
$d \ge \Omega(\log^2 n)$ is necessary. This implies a super-polynomial lower bound on the alphabet size $2^d$ in terms of $n$, which
is one of their main results.

We improve on both upper and lower bounds and show that $d$ is linear in $n$. We note that our construction improves upon the random construction,
which for us was somewhat surprising. For convenience we describe it when $n$ is a power of two, but note that it holds for any $n$ with minimal modifications.

\begin{theorem}[Explicit construction]\label{thm:construct}
Let $n$ be a power of two. There exists $\gamma:[n] \times [n] \to \F_2^{d}$ for $d=3n$ which is simple cycle free.
\end{theorem}

Our main technical result is a nearly matching lower bound.

\begin{theorem}[Lower bound]\label{thm:lower_bound_dim}
Let $\gamma:[n] \times [n] \to \F_2^d$ be simple cycle free. Then $d \ge n/2-2$.
\end{theorem}

\subsection{Labeling by general Abelian groups}
The definition of simple cycles free labeling can be extended to labeling by general Abelian groups, not just $\F_2^d$.
Let $H$ be an Abelian group, and let $\gamma:[n] \times [n] \to H$.
We say that $\gamma$ is simple cycle free if for any simple cycle $C$,
$$
\sum_{e \in C} \sign(e) \gamma(e) \ne 0.
$$
where $\sign(e) \in \{-1,1\}$ is an alternating sign assignment to the edges of $C$ (these are sometimes called circulations).
We note that the analysis of Gopalan \etal~\cite{gopalan2017maximally} can be extended to non-binary alphabets $\F_p$,
in which case their combinatorial characterization extends to the one above with $H=\F_p$.

\begin{theorem}\label{thm:lower_bound_dim2}
Let $H$ be an Abelian group. Let $\gamma:[n] \times [n] \to H$ be simple cycle free. Then $|H| \ge 2^{n/2-2}$.
\end{theorem}

As a side remark, we note that the study of graphs with nonzero circulations was instrumental in the recent construction of a deterministic
quasi-polynomial algorithm for perfect matching in NC~\cite{fenner2016bipartite}. However, beyond some superficial similarities,
the setup seems inherently different than ours. For starters, they study general bipartite graphs, while we study the complete graphs. Moreover, they need to handle certain families of cycles,
not necessarily simple,  while in this work we focus on simple cycles.

The proofs of Theorem~\ref{thm:lower_bound_dim} and Theorem~\ref{thm:lower_bound_dim2} rely on the study of
a certain Cayley graph of the permutation group, which encodes the property of simple cycle free labeling. Surprisingly, the 
corresponding graph is the Birkhoff polytope graph.

\subsection{The Birkhoff polytope graph}
Let $S_n$ denote the symmetric group of permutations on $[n]$.
A permutation $\tau \in S_n$ is said to be a \emph{cycle} if, except for its fixed points, it contains a single non-trivial cycle (in particular, the identity is not a cycle).
We denote by $\CC_n \subset S_n$ the set of cycles. The Cayley graph $\B_n=\cay(S_n, \CC_n)$ is a graph with vertex set $S_n$ and edge set $\{(\pi, \tau \pi): \pi \in S_n, \tau \in \CC_n\}$. Note that this graph is undirected, as if $\tau \in \CC_n$ then also $\tau^{-1} \in \CC_n$.

The graph $\B_n$ turns out to be widely studied: it is the graph of the Birkhoff polytope, which is the convex hull of all $n \times n$ permutation matrices.
See for example~\cite{billera1994combinatorics} for a proof. Our analysis does not use this connection; we use the description of the graph as a Cayley graph.

The following claim shows that Theorem~\ref{thm:lower_bound_dim2} reduces to bounding the size of the largest independent set in the Birkhoff polytope graph.

\begin{claim}
\label{claim:reduce}
Let $H$ be an Abelian group.
Assume that $\gamma:[n] \times [n] \to H$ is simple cycle free. Then $\B_n$ contains an independent set of size $\ge n! / |H|$.
\end{claim}

\begin{proof}
Define
$$
A = \left\{\pi \in S_n: \sum_{i=1}^n \gamma(i,\pi(i))=h\right\},
$$
where $h \in H$ is chosen to maximize the size of $A$. Thus $|A| \ge n!/|H|$. We claim that $A$ is an independent set in $\B_n$.

Assume not. Then there are two permutations $\pi,\pi' \in A$ such that $\tau = \pi (\pi')^{-1} \in \CC_n$. Let $M_{\pi}=\{(i, \pi(i)): i \in [n]\}$ denote the matching
in $K_{n,n}$ associated with $\pi$, and define $M_{\pi'}$ analogously. Let $C = M_{\pi} \oplus M_{\pi'}$ denote their symmetric difference. The fact that $\tau \in \CC_n$ has
exactly one cycle, is equivalent to $C$ being a simple cycle. Let $\sign(\cdot)$ be an alternating sign assignment to the edges of $C$. Then
$$
\sum_{e \in C} \sign(e) \gamma(e) = \sum_{e \in M_{\pi}} \gamma(e) - \sum_{e \in M_{\pi'}} \gamma(e) = h-h=0.
$$
This violates the assumption that $\gamma$ is simple cycle free.
\end{proof}

The construction of a simple cycle free labeling in Theorem~\ref{thm:construct}, combined with Claim~\ref{claim:reduce}, implies that the Birkhoff polytope graph
contains a large independent set.
\begin{corollary}\label{cor:lower_bound_indep}
Let $n$ be a power of two. Then $\B_n$ contains an independent set of size $\ge n!/9^n$.
\end{corollary}

We also give in the appendix a construction of a larger independent set in the Birkhoff polytope graph, not based on a simple cycle free labeling.
\begin{theorem}\label{thm:better_example}
Let $n$ be a power of two. Then $\B_n$ contains an independent set of size $\ge n!/4^n$.
\end{theorem}

The best previous bounds we are aware of are by Onn~\cite{onn1993geometry} who proved that $\B_n$ contains an independent set of size $\ge n^{\Omega(\sqrt{n})}$.

Our main technical result is an upper bound on the largest size of an independent set in the Birkhoff polytope graph.

\begin{theorem}\label{thm:lower_bound_perm}
The largest independent set in $\B_n$ has size $\le n! / 2^{(n-4)/2}$.
\end{theorem}

As a side remark, we note that general bounds on the independence number of graphs, such as the Hoffman bound, give much weaker bounds.
A standard application of the Hoffman bound gives a much weaker bound for the independence number of $\B_n$ of $O(n!)$;
and if we restrict all permutations to have the same sign, the bound improves to $O((n-1)!)$. The reason is that 
the Hoffman bounds (at least in its simplest form) directly relates to the minimal eigenvalues of the graph. However,
in our case the eigenvalues are controlled by the irreducible representations of $S_n$, and the extreme eigenvalues
are given by low dimensional representations. This prohibits obtaining strong bounds on the independence number directly.

In order to overcome this barrier, our analysis circumvents the effect of the low dimensional representations by appealing to a
structure vs. randomness dichotomy specialized for our setting. It allows us to either reduce the dimension of the ambient
group, or restrict to pseudo-random assumptions about the actions of the low dimensional representations.

\paragraph{Organization.}
We prove Theorem~\ref{thm:construct} in Section~\ref{sec:lower} and
Theorem~\ref{thm:lower_bound_perm} in Section~\ref{sec:upper}. Theorem~\ref{thm:better_example} is proved in Appendix~\ref{sec:better_example}.

\paragraph{Acknowledgements.}
We thank Ran Gelles and Sergey Yekhanin for useful discussions on the problem and comments on a preliminary version of this paper.
We thank Igor Pak for bringing to our attention that the Cayley graph which we study is the Birkhoff polytope graph.

\section{A construction of a simple cycle free labeling}
\label{sec:lower}

We prove Theorem~\ref{thm:construct} in this section. We first introduce some notation. For $x \in [n]$
denote by $e^n_{x} \in \F_2^n$ the unit vector with $1$ in coordinate $x$ and $0$ in all other coordinates. We let $0^n \in \F_2^n$ denote
the all zero vector.

Let $n$ be a power of two.  We define recursively a labeling $\gamma_n:[n] \times [n] \to \F_2^{3n}$. For $n=2$ set (for example)
$$
\gamma_2(0,0)=e_1^6, \gamma_2(0,1)=e_2^6, \gamma_2(1,0)=e_3^6, \gamma_2(1,1)=e_4^6.
$$
Assume $n>2$. Let $x' = x \mod (n/2)$ and $y' = y \mod (n/2)$, where $x',y' \in [n/2]$. Define $\gamma_n(x,y) \in \F_2^{3n}$ recursively as
\begin{enumerate}
\item[(i)] The first $n$ bits of $\gamma_n(x,y)$ are $e^n_x$ if $y \le n/2$, and otherwise they are $0^n$.
\item[(ii)] The next $n/2$ bits of $\gamma_n(x,y)$ are $e^{n/2}_{y'}$ if $x \le n/2$, and otherwise they are $0^{n/2}$.
\item[(iii)] The last $3n/2$ bits of $\gamma_n(x,y)$ are defined recursively to be $\gamma_{n/2}(x',y')$.
\end{enumerate}

We claim that $\gamma_n$ is indeed simple cycle free. For $n=2$ it is simple to verify this directly, so assume $n>2$.

Let $C$ be a simple cycle in $K_{n,n}$, and assume towards a contradiction that $\sum_{e \in C} \gamma_n(e)=0$. Assume $C$ has $2k$ nodes, for some $2 \le k \le n$,
and let these be $C=(x_1,y_1,x_2,y_2,\ldots,x_k,y_k,x_1)$. We denote $X=\{x_1,\ldots,x_k\}$ and $Y=\{y_1,\ldots,y_k\}$. Define furthermore $L=\{1,\ldots,n/2\}$ and $U=\{n/2+1,\ldots,n\}$.

\begin{claim}\label{claim:Y_subset}
Either $Y \subset L$ or $Y \subset U$.
\end{claim}

\begin{proof}
Assume that both $Y \cap L$ and $Y \cap U$ are nonempty. Then there must exist $i \in [k]$ with $y_i \in L$ and
$y_{i+1} \in U$, where if $i=k$ then we take the subscript modulo $k$. Recall that $x_{i+1}$ is the neighbour of $y_i,y_{i+1}$ in $C$. Its contribution
to the first $n$ bits of the sum is $e_{x_{i+1}}^n$, since $y_i \le n/2$ and $y_{i+1}>n/2$. Note that no other edge in $C$ has a nonzero value in coordinate $x_{i+1}$.
Thus the $x_{i+1}$ coordinate in the sum over $C$ is $1$, which contradicts the assumption that the sum over $C$ is zero.
\end{proof}

Thus we can assume from now on that either $Y \subset L$ or $Y \subset U$.

\begin{claim}
Either $X \subset L$ or $X \subset U$.
\end{claim}

\begin{proof}
Assume that $Y \subset L$, and the case of $Y \subset U$ is identical.
Assume that both $X \cap L$ and $X \cap U$ are both nonempty. Then there must exist $i \in [k]$ with $x_i \in L$ and
$x_{i+1} \in U$. Recall that $y_{i}$ is the neighbour of $x_i,x_{i+1}$ in $C$. Its contribution
to the 2nd batch (of $n/2$ bits) of the sum is $e_{y'_{i}}^{n/2}$, since $x_i \le n/2$ and $x_{i+1}>n/2$.
Note that no other edge in $C$ has a nonzero value in coordinate $n+y'_{i}$, where we here we need the assumption that $Y \subset L$ or
$Y \subset U$.
Thus the $n+y'_{i}$ coordinate in the sum over $C$ is $1$, which contradicts the assumption that the sum over $C$ is zero.
\end{proof}

Thus we have that $X \subset U$ or $X \subset L$, and similarly $Y \subset U$ or $Y \subset L$.
Thus, $C$ is a simple cycle in $K_{n/2,n/2}$ embedded in $K_{n,n}$ in one of four disjoint ways: $L \times L$, $L \times U$, $U \times L$ or $U \times U$.
Observe that in each of these copies, the last $3n/2$ coordinates of the sum are precisely $\gamma_{n/2}$, so by induction $C$ cannot have zero sum.

\section{The independence number of the Birkhoff polytope graph}
\label{sec:upper}

We prove Theorem~\ref{thm:lower_bound_perm} in this section.
Let $A$ be an independent set in $\B_n$. We prove an upper bound on the size of $A$. Concretely, we will show that $|A| \le \frac{a}{c^n} n!$ for some absolute
constants $a,c>1$. As we will see at the end, the choice of $a=4, c=\sqrt{2}$ works.

The proof relies on representation theory, in particular representation theory of the symmetric group. We
refer readers to the excellent book of Sagan~\cite{sagan2013symmetric}, which provides a thorough introduction to the topic. We will try to adhere to the notations
in that book whenever possible.

\paragraph{Overall Strategy.} Our basic plan will be to break our analysis into two cases based on whether or not the action of $A$ on $m$-tuples is nearly uniform for all $m$. This will be in analogy with standard structure vs. randomness arguments. If the action on $m$-tuples is highly non-uniform, this will allow us to take advantage of this non-uniformity to reduce to a lower-dimensional case. On the other hand, if $A$ acts nearly uniformly on $m$-tuples, this suggests that it behaves somewhat randomly. This intuition can be cashed out usefully by considering the Fourier-analytic considerations of this condition, which will allow us to prove that some pair of elements of $A$ differ by a simple cycle using Fourier analysis on $S_n$.

\paragraph{Non-Uniform Action on Tuples.}
Let $[n]_m=\{(i_1,\ldots,i_m): i_1,\ldots,i_m \in [n] \text{ distinct}\}$ denote the family of ordered $m$-tuples of distinct elements of $[n]$.
Its size is $(n)_m = n(n-1)\cdots(n-m+1)$. A permutation $\pi \in S_n$ acts on $[n]_m$ by sending $I=(i_1,\ldots,i_m)$ to $\pi(I)=(\pi(i_1),\ldots,\pi(i_m))$.
Below when we write $\Pr_{\pi \in A}[\cdot]$ we always mean the probability of an event under a uniform choice of $\pi \in A$.

Notice that if $\Pr_{\pi \in A}[\pi(I)=J] \ge c^m / (n)_m$ for some pair $I,J\in [n]_m$, this will allow us to reduce to a lower dimensional version of the problem. In particular, if we let $A'=\{\pi\in A: \pi(I)=J\}$, we note that $|A| \leq |A'|(n)_m / c^m$. On the other hand, after multiplying on the left and right by appropriate permutations (an operation which doesn't impact our final problem), we can assume that $I=J=\{n-m+1,\ldots,n\}$. Then, if $A$ were an independent set for $\B_n$, $A'$ would correspond to an independent set for $\cay(S_{n-m}, \CC_{n-m})$. Then, if we could prove the bound that $|A'| \le \frac{a}{c^{n-m}} (n-m)!$, we could inductively prove that $|A| \le \frac{a}{c^n} n!$.

\paragraph{Uniform Action on Tuples.}
When the action of $A$ on $m$-tuples is near uniform for all $m$, we will attempt to show that two elements of $A$ differ by a simple cycle using techniques from the Fourier analysis of $S_n$. In fact, we will show the stronger statement that some pair of elements of $A$ differ by a single cycle of length $n$.

Some slight complications arise here when parity of the permutations here is considered. In particular, all $n$-cycles have the same parity. This is actually a problem for $n$ even, as all such cycles will be odd, and thus our statement will fail if $A$ consists only of permutations with the same parity. Thus, we will have to consider our statement only in the case of $n$ odd. Even in this case though, parity will still be relevant. In particular, note that the difference between two permutations in $A$ can be a cycle of length $n$ only if the initial permutations had the same parity. Thus, we lose very little by restricting our attention to only elements of $A$ with the more common parity. This will lose us a factor of $2$ in the size of $A$, but will make our analysis somewhat easier. We are now prepared to state our main technical proposition:

\begin{proposition}\label{prop:main}
Let $n$ be an odd integer and let $c>1$ be a sufficiently small constant. Let $A\subset S_n$ be a set of permutations satisfying:
\begin{enumerate}
\item[(i)] All elements of $A$ are of the same sign.
\item[(ii)]  For any even $m<n$ and any $I,J\in [n]_m$, $\Pr_{\pi \in A}[\pi(I)=J]<\frac{c^m}{(n)_m}$.
\end{enumerate}
Then there exist two elements of $A$ that differ by a cycle of length $n$.
In particular, we can take $c=\sqrt{2}$.
\end{proposition}

\paragraph{Remark.} In the second condition above, we consider only even $m$. This is because if this condition fails, we are going to use our other analysis to recursively consider permutations of $[n-m]$, and would like $n-m$ to also be odd.

We prove Proposition~\ref{prop:main} below, and then show that it implies Theorem~\ref{thm:lower_bound_perm}.

\begin{proof}
First, note that by replacing all $\pi\in A$ by $\pi \sigma$ for some odd permutation $\sigma$ if necessary, it suffices to assume that all $\pi\in A$ are even. We will assume this henceforth.

\paragraph{Rephrasing the problem using class functions.} Let $\CC'_n$ denote the set of $n$-cycles in $S_n$.
Define two class functions $\varphi,\psi \in \R[S_n]$ as
$$
\varphi = \frac{1}{|S_n| |A|^2} \sum_{\sigma \in S_n, \pi,\pi' \in A} \sigma \pi (\pi')^{-1} \sigma^{-1}, \qquad
\psi = \frac{1}{|\CC'_n|} \sum_{\tau \in \CC'_n} \tau.
$$
It is easy to see that our conclusion is equivalent to showing that $\ip{\varphi,\psi}>0$.

Let $\lambda \vdash n$ denote a partition of $n$, namely $\lambda=(\lambda_1,\ldots,\lambda_k)$ where $\lambda_1 \ge \ldots \ge \lambda_k \ge 1$ and $\sum \lambda_i=n$.
The irreducible representations of $S_n$ are the Specht modules, which are indexed by partitions $\{S^{\lambda}: \lambda \vdash n\}$. Let $\chi^{\lambda}:S_n \to \R$ denote their
corresponding characters. Their action extends linearly to $\R[S_n]$. Namely, if $\zeta \in \R[S_n]$ is given by
$\zeta= \sum_{\pi \in S_n} \zeta_{\pi} \pi \in \R[S_n]$ where $\zeta_{\pi} \in \R$ then
$\chi^{\lambda}(\zeta) = \sum_{\pi \in S_n} \zeta_{\pi} \chi^{\lambda}(\pi)$.

As $\varphi, \psi \in \R[S_n]$ are class functions, their inner product equals
\begin{equation}\label{eq:ip1}
\ip{\varphi, \psi} = \sum_{\lambda \vdash n} \chi^{\lambda}(\varphi) \chi^{\lambda}(\psi).
\end{equation}
Let $(n) \in \CC'_n$ be a fixed cycle of length $n$. As all elements in $\psi$ are conjugate to $(n)$,
we have $\chi^{\lambda}(\psi) = \chi^{\lambda}((n))$ and we can simplify Equation~\eqref{eq:ip1} to
\begin{equation}\label{eq:ip2}
\ip{\varphi, \psi} = \sum_{\lambda \vdash n} \chi^{\lambda}(\varphi) \chi^{\lambda}((n)).
\end{equation}
Thus, we are lead to explore the action of the irreducible characters on the full cycle $(n)$.

\paragraph{Characters action on the full cycle.}
The Murnaghan-Nakayama rule is a combinatorial method to compute the value of a character $\chi^{\lambda}$ on a conjugacy class, which in our case is $(n)$.
In this special case it is very simple. It equals zero unless $\lambda$ is a hook, e.g. its corresponding tableaux has only one row and one column, and otherwise its either $-1$ or $1$.
Concretely, let $h_m = (n-m,1,1,\ldots,1)$ for $0 \le m \le n-1$ denote the partition corresponding to a hook. Then
\begin{equation}\label{eq:MN}
\chi^{\lambda}((n)) = \bigg \{
\begin{array}{ll}
(-1)^m & \text{if } \lambda=h_m\\
0 & \text{otherwise}
\end{array}.
\end{equation}
Thus we can simplify Equation~\eqref{eq:ip2} to
\begin{equation}\label{eq:ip3}
\ip{\varphi, \psi} = \sum_{m=0}^{n-1} (-1)^m \chi^{h_m}(\varphi).
\end{equation}

\paragraph{Bounding the characters on $\varphi$.}
The character $h_0$ corresponds to the trivial representation, and by our definition of $\varphi$ it equals $\chi^{h_0}(\varphi)=1$.
Observe that we can simplify $\chi^{\lambda}(\varphi)$ as
\begin{equation}\label{eq:chi_varphi}
\chi^{\lambda}(\varphi) = \frac{1}{|A|^2 |S_n|} \sum_{\pi,\pi' \in A, \sigma \in S_n} \chi^{\lambda}(\sigma \pi (\pi')^{-1} \sigma^{-1})
= \frac{1}{|A|^2} \sum_{\pi,\pi' \in A} \chi^{\lambda}(\pi (\pi')^{-1}).
\end{equation}

First, we argue that the evaluation of characters on $\varphi$ is always nonnegative.
\begin{claim}
\label{claim:nonnegative}
$\chi^{\lambda}(\varphi) \ge 0$ for all $\lambda \vdash n$.
\end{claim}
\begin{proof}
Let $\zeta \in \R[S_n]$ be given by $\zeta = \frac{1}{|A|} \sum_{\pi \in A} \pi$.
Then
$$
\chi^{\lambda}(\varphi) = \frac{1}{|A|^2} \sum_{\pi,\pi' \in A} \Tr \left(S^{\lambda}(\pi) S^{\lambda}((\pi')^{-1}) \right) = \Tr \left(S^{\lambda}(\zeta) S^{\lambda}(\zeta)^T\right)
= \|S^{\lambda}(\zeta)\|_F^2,
$$
where for a matrix $M$ its Frobenius norm is given by $\|M\|_F^2 = \sum |M_{i,j}|^2$. In particular it is always nonnegative.
\end{proof}

The following lemma bounds $\chi^{h_m}(\varphi)$. Observe that in particular for $c=1$ it gives $\chi^{h_m}(\varphi)=0$. However, we would use it to
obtain effective bounds when $c>1$.

\begin{lemma}
\label{lemma:bound_chi}
Let $m \in \{1,\ldots,n-1\}$. For any
even $k \in \{m,\ldots,n\}$ it holds that $\chi^{h_m}(\varphi) \le \frac{c^{k}-1}{{k \choose m}}$.
\end{lemma}

\begin{proof}
Let $M^{\mu}$ denote the (not irreducible) Young module associated with a partition $\mu \vdash n$.
In the case of $\mu = h_k$ it corresponds to the action of $S_n$ on $[n]_k$. That is, for any $\pi \in S_n$ we have that $M^{h_k}(\pi)$ is a
matrix whose rows and columns are indexed by $I,J \in [n]_k$ respectively, where $M^{h_k}(\pi)_{I,J}=1_{\pi(I)=J}$. Observe that
$M^{h_k}(\pi^{-1})=\left(M^{h_k}(\pi)\right)^T$.
We extend this action to $\R[S_n]$ linearly.

Recall that $\zeta = \frac{1}{|A|} \sum_{\pi \in A} \pi \in \R[S_n]$. By assumption (ii) in Proposition~\ref{prop:main} we have
$$
\left( M^{h_k}(\zeta) \right)_{I,J} = \Pr_{\pi \in A}[\pi(I)=J] \le \frac{c^{k}}{(n)_k}.
$$
Thus, we can bound the Frobenius norm of $M^{h_k}(\zeta)$ by
$$
\| M^{h_k}(\zeta)\|_F^2 = \sum_{I,J} |\left( M^{h_k}(\zeta) \right)_{I,J}|^2 \leq \left(\frac{c^{k}}{(n)_k} \right)\sum_{I,J} |\left( M^{h_k}(\zeta) \right)_{I,J}| = c^k.
$$
This is useful as
$$
\Tr(M^{h_k} (\varphi)) = \Tr \left( M^{h_k}(\zeta) \left(M^{h_k}(\zeta) \right)^T \right) = \| M^{h_k}(\zeta)\|_F^2 \le c^{k}.
$$

The Kostka numbers $K_{\lambda, \mu}$ denote the multiplicity of the Specht module $S^{\lambda}$ in the Young module $M^{\mu}$.
We can thus decompose
$$
\Tr(M^{\mu} (\varphi)) = \sum_{\lambda} K_{\lambda,\mu} \chi^{\lambda}(\varphi).
$$
We saw that $\chi^{\lambda}(\varphi) \ge 0$ for all $\lambda$.
By Young's rule, $K_{\lambda, \mu}$ equals the number of semistandard tableaux of shape $\lambda$ and content $\mu$. In particular, it is always a nonnegative integer.
In the special case of $\lambda=h_m$ and $\mu=h_k$ for $k \ge m$, Young's rule is simple to compute and gives
$$
K_{h_m, h_k} = {k \choose m}.
$$
Recall that $\chi^{h_0}$ is the trivial representation, for which $K_{h_0, h_k}=1$ and $\chi^{h_0}(\varphi)=1$. Thus
$$
1+{k \choose m} \chi^{h_m}(\varphi) \le \sum_{\lambda} K_{\lambda,h_k} \chi^{\lambda}(\varphi) = \Tr(M^{h_k} (\varphi)) \le c^{k}.
$$
\end{proof}

We next apply Lemma~\ref{lemma:bound_chi} to bound $\chi^{h_m}(\varphi)$ for all $1 \le m \le n-1$.
If $m \le n/2$ then we can apply Lemma~\ref{lemma:bound_chi} for $k=2m$ and obtain the bound
$$
\chi^{h_m}(\varphi) \le \frac{c^{2m}-1}{{2m \choose m}}.
$$
For $m > n/2$ we need the following claim, relating $\chi^{h_m}$ to $\chi^{h_{n-1-m}}$.

\begin{claim}
\label{claim:dual}
For any $1 \le m \le n-1$ it holds that $\chi^{h_m}(\varphi) = \chi^{h_{n-1-m}}(\varphi)$.
\end{claim}

\begin{proof}
For any partition $\lambda$ let $\lambda'$ denote the transpose (also known as conjugate) partition. It satisfies $\chi^{\lambda'}(\pi) = \chi^{\lambda}(\pi) \sign(\pi)$ for all $\pi \in S_n$,
where $\sign:S_n \to \{-1,1\}$ is the sign representation. As all elements in $A$ are even permutations, it holds
by the definition of $\varphi$ that
$$
\chi^{\lambda'}(\varphi) = \frac{1}{|A|^2} \sum_{\pi,\pi' \in A} \chi^{\lambda'} (\pi (\pi')^{-1}) = \frac{1}{|A|^2} \sum_{\pi,\pi' \in A} \chi^{\lambda} (\pi (\pi')^{-1}) =
\chi^{\lambda}(\varphi).
$$
In particular if $\lambda=h_m$ then $\lambda'=h_{n-1-m}$.
\end{proof}

Next, we lower bound $\ip{\varphi, \psi'}$ as follows. The dominant terms are $\chi^{h_0}(\varphi)=\chi^{h_{n-1}}(\varphi)=1$. For any $1 \le m \le (n-1)/2-1$, the corresponding
term in Equation~\eqref{eq:ip3} appears twice, once as $(-1)^m \chi^{h_m}(\varphi)$ and once as $(-1)^{n-1-m} \chi^{h_{n-1-m}}(\varphi)=(-1)^m \chi^{h_m}(\varphi)$.
The term for $m=(n-1)/2$ appears once.

Furthermore, as $\chi^{h_m}(\varphi) \ge 0$ for all $m$ by Claim~\ref{claim:nonnegative}, the only negative terms correspond to odd $1 \le m \le (n-1)/2$. Thus we can lower bound
\begin{equation}\label{eq:ip_lb}
\frac{1}{2} \ip{\varphi, \psi'} \ge 1 - \sum_{m \ge 1,\; m \text{ odd}} \frac{c^{2m}-1}{{2m \choose m}}.
\end{equation}
It is not hard to show that this is positive if $c>1$ is small enough. If we take $c=\sqrt{2}$, the right hand side of Equation \eqref{eq:ip_lb} is slightly negative for large enough $m$ (the limit as $m\rightarrow\infty$ is $-0.02451...$). However, when $n\geq 8$, the second term can be replaced by $\frac{c^8-1}{\binom{8}{3}}$ rather than $\frac{c^6-1}{\binom{6}{3}}$, making our lower bound on $\frac{1}{2} \ip{\varphi, \psi'}$ at least $0.057$. This completes our proof.

\end{proof}

We are now prepared to prove Theorem \ref{thm:lower_bound_perm}.
\begin{proof}
We first prove that if $n$ is odd and if all permutations in $A$ have the same sign, then $|A|\leq \frac{n!}{2^{(n-1)/2}}.$

We proceed by induction on $n$. Firstly, we note that if $n=1$, the bound follows trivially.

For odd $n>1$, we note that unless there is some even $m<n$ and some $I,J\in [n]_m$ with $\Pr_{\pi\in A}[\pi(I)=J] \ge 2^{m/2}/(n)_m$, then our result follows immediately from Proposition \ref{prop:main}. Otherwise, we may assume without loss of generality that $I=J=(n-m+1,\ldots,n)$. It then follows that letting $A'=\{\pi\in A: \pi(I)=J\}$, we can think of $A'$ as a set of permutations on $[n-m]$. Also, note that $A$ being an independent set for $\B_n$, implies that $A'$ is an independent set for $\cay(S_{n-m}, \CC_{n-m})$. Therefore, by the inductive hypothesis:
$$
|A| \leq (n)_m2^{-m/2}|A'| \leq (n)_m2^{-m/2} (n-m)!/2^{(n-m-1)/2} = n!/2^{(n-1)/2}.
$$

We now need to reduce to the case of $n$ odd and $A$ consisting only of permutations of the same sign. First, restricting $A$ to only permutations of the most common sign, we can assume that all permutations in $A$ have the same sign, losing only a factor of $2$ in $|A|$. Now, if $n$ is odd, we are done. otherwise, let $j$ be the most likely value of $\pi(n)$ for $\pi$ taken from $A$. We have that $\Pr_{\pi\in A}[\pi(n)=j] \geq 1/n$. Without loss of generality, $j=n$ and we can let $A'=\{\pi\in A:\pi(n)=n\}$. Since $A'$ is an independent set in $\cay(S_{n-1}, \CC_{n-1})$, and since $n-1$ is odd, we have
$$
|A| \leq n|A'| \leq n (n-1)!/2^{(n-2)/2} = n!/2^{n/2-1}.
$$
\end{proof}

\bibliographystyle{abbrv}
\bibliography{no_zero_cycles}

\appendix

\section{A construction of a larger independent set}
\label{sec:better_example}

We prove Theorem~\ref{thm:better_example} in this section.
Assume that $n=2^m$. We construct $A \subset S_n$ of size $|A| \ge n! / 4^n$, such that $A$ is an independent set in $\B_n$.

Let $T_{i,j} = \{2^{m-i}(j-1)+1,\ldots,2^{m-i}j\}$ for $0 \le i \le m, 1 \le j \le 2^i$.
Note that $\{T_{i,j}: j \in [2^i]\}$ is a partition of $[n]$ for every $i$, that $|T_{i,j}|=2^{m-i}$ and that $T_{i,2j-1} \cup T_{i,2j}$ is a partition of $T_{i-1,j}$.

We define a sequence of subsets of $S_n$. For $1 \le i \le m$ let $M_i = {2^{m-i+1} \choose 2^{m-i}}$. For any set $R$ of size $|R|=2^{m-i+1}$ let
$\ind_i(R,\cdot)$ be a bijection between subsets of $R$ of size $2^{m-i}$ and $\Z_{M_i}$. Define $A_0 = S_n$ and
$$
A_i = \left\{\pi \in A_{i-1}: \sum_{j=1}^{2^{i-1}} \ind_{i}(\pi(T_{i-1,j}), \pi(T_{i,2j-1})) \equiv 0 \mod M_i\right\}.
$$
Since each value mod $M_i$ occurs equally often as a $\ind_{i}(\pi(T_{i-1,j}), \pi(T_{i,2j-1}))$ for each $j$, and since these values are independent of one another, $|A_i| = |A_{i-1}| / M_i$.
Finally set $A=A_m$. The following claim (applied for $i=m$) shows that $A$ is an independent set in $\B_n$.

\begin{claim}
Let $1 \le i \le m$. Let $\pi,\pi' \in A_i$ be such that $\tau = \pi (\pi')^{-1} \in \CC_n$. Then there exists $j_i \in [2^i]$ such that
\begin{enumerate}
\item $\tau(T_{i,j_i})=T_{i,j_i}$.
\item $\tau(x)=x$ for all $x \in T_{i,j}, j \ne j_i$.
\end{enumerate}
\end{claim}

\begin{proof}
We prove the claim by induction on $i$. The case of $i=1$ follows from the definition of $A_1$. By assumption $\pi,\pi'$ fix both $T_{1,1}$ and $T_{1,2}$.
However, as $\tau=\pi (\pi')^{-1}$ is a cycle, it must be contained in either $T_{1,1}$ or $T_{1,2}$. This implies that $\tau(x)=x$ for all $x \in T_{1,1}$ or all $x \in T_{1,2}$.

Consider next the case of $i>1$. By induction $\pi(T_{i-1,j})=\pi'(T_{i-1,j})$ for all $j \in [2^{i-1}]$. Moreover, there exists $j'=j_{i-1}$ such that
$\pi(x)=\pi'(x)$ for all $x \in T_{i-1,j}, j \ne j'$. This implies that $\pi(T_{i,j})=\pi'(T_{i,j})$ for all $j \not\in \{2j'-1,2j'\}$.

Next, the assumption that $\pi,\pi' \in A_i$ guarantees that
$$
\sum_{j=1}^{2^{i-1}} \ind_i(\pi(T_{i-1,j}), \pi(T_{i,2j-1})) \equiv \sum_{j=1}^{2^{i-1}} \ind_i(\pi'(T_{i-1,j}),\pi'(T_{i,2j-1})) \equiv 0 \mod M_{i}.
$$
For any $j \ne j'$ we know that $\pi(T_{i-1,j})=\pi'(T_{i-1,j})$ and $\pi(T_{i,2j-1})=\pi'(T_{i,2j-1})$, so $\ind_i(\pi(T_{i-1,j}),\pi(T_{i,2j-1}))=\ind_i(\pi'(T_{i-1,j}),\pi'(T_{i,2j-1}))$.
Thus we obtain that also $\ind_i(\pi(T_{i-1,j'}),\pi(T_{i,2j'-1}))=\ind_i(\pi'(T_{i-1,j'}),\pi'(T_{i,2j'-1}))$. Moreover, as we also know that $\pi(T_{i-1,j'})=\pi'(T_{i-1,j'})$ and that $\ind_i(\pi(T_{i-1,j'}),\cdot)$ is a bijection to $\Z_{M_i}$, it must be the case that $\pi(T_{i,2j'-1})=\pi'(T_{i,2j'-1})$ and hence also $\pi(T_{i,2j'})=\pi'(T_{i,2j'})$.
Thus we conclude that $\pi(T_{i,j})=\pi'(T_{i,j})$ for all $j \in [2^i]$.

To conclude, as $\tau=\pi (\pi')^{-1}$ is a cycle, it must be contained in either $T_{i,2j'-1}$ or $T_{i,2j'}$. Thus, $\tau$ must fix all points in $T_{i,2j'-1}$ or all points in $T_{i,2j'}$. We set $j_i \in \{2j'-1, 2j'\}$ accordingly.
\end{proof}

Finally, we compute the size of $A$. As $|A_i| = |A_{i-1}| / M_i$ and $M_i = {2^{m-i+1} \choose 2^{m-i}} \le 2^{2^{m-i+1}}$
we obtain that
$$
|A| \ge \frac{n!}{\prod_{i=1}^m 2^{2^i}} \ge \frac{n!}{2^{2^{m+1}}}=\frac{n!}{4^n}.
$$

\end{document}